\documentclass[oneside,reqno,11pt]{amsart}
\usepackage{epsfig}
\usepackage[dvipsnames]{xcolor}
\usepackage{amssymb,amsmath,amsthm,amstext,amsfonts}
\usepackage[shortlabels]{enumitem}
\usepackage[normalem]{ulem}
\usepackage{amsmath,amscd,amstext,amsthm,amsfonts,latexsym}
\usepackage[colorlinks=true,linkcolor=blue, urlcolor=blue, citecolor=blue,backref]{hyperref}
\hypersetup{
 citecolor=blue}

\makeatletter

\theoremstyle{definition}

\newtheorem*{thm*}{Theorem}

\makeatletter
\numberwithin{equation}{section}
\numberwithin{figure}{section}
\theoremstyle{plain}
\newtheorem*{cor*}{\protect\corollaryname}
\theoremstyle{plain}
\newtheorem{thm}{\protect\theoremname}[section]
\theoremstyle{definition}

\theoremstyle{remark}

\theoremstyle{plain}

\theoremstyle{plain}
\newtheorem{lem}[thm]{\protect\lemmaname}
\theoremstyle{plain}
\newtheorem{cor}[thm]{\protect\corollaryname}

\usepackage[utf8]{inputenc}
\usepackage[T1]{fontenc}
\usepackage{lmodern}

\usepackage{geometry}
\usepackage{amsmath}

\numberwithin{equation}{section}
\numberwithin{figure}{section}
\usepackage{enumitem}		
 \let\footnote=\endnote
\@ifundefined{lettrine}{\usepackage{lettrine}}{}

\theoremstyle{definition}

\def\k{\kappa}

\def\R{\mathbb{R}}
\def\C{\mathbb{C}}

\def\N{\mathbb{N}}
\def\Z{\mathbb{Z}}

\usepackage{float}

\subjclass[2000]{}

\def\Si{\Sigma}

\def\I{\mathrm{I}}

\def\R{\mathbb{R}}

\def\C{\mathbb{C}}

\def\A{\mathcal{A}}
\def\B{\mathcal{B}}

\def\C{\mathcal{C}}

\newcommand\e{\mathbf e}
\def\M{\mathcal{M}_{\text{inv}}}

\def\Sig{\Sigma}
\def\glr{\text{GL}_d(\R)}
\def\gl2{\text{GL}_2(\R)}
\def\sl2{\text{SL}_2(\R)}

\makeatother
\usepackage[T1]{fontenc}
\usepackage[utf8]{inputenc}

\renewcommand{\le}{\leqslant}
\renewcommand{\leq}{\leqslant}
\renewcommand{\geq}{\geqslant}
\renewcommand{\ge}{\geqslant}

\usepackage[toc]{appendix}

  \providecommand{\corollaryname}{Corollary}
  \providecommand{\definitionname}{Definition}
  \providecommand{\lemmaname}{Lemma}
  \providecommand{\propositionname}{Proposition}
  \providecommand{\remarkname}{Remark}
  \providecommand{\theoremname}{Theorem}
\providecommand{\theoremname}{Theorem}

\usepackage{tikz,xcolor,hyperref}

\definecolor{lime}{HTML}{A6CE39}
\DeclareRobustCommand{\orcidicon}{
	\begin{tikzpicture}
	\draw[lime, fill=lime] (0,0) 
	circle [radius=0.16] 
	node[white] {{\fontfamily{qag}\selectfont \tiny ID}};
	\draw[white, fill=white] (-0.0625,0.095) 
	circle [radius=0.007];
	\end{tikzpicture}
	\hspace{-2mm}
}
\foreach \x in {A, ..., Z}{\expandafter\xdef\csname orcid\x\endcsname{\noexpand\href{https://orcid.org/\csname orcidauthor\x\endcsname}
			{\noexpand\orcidicon}}
}

\renewcommand{\le}{\leqslant}
\renewcommand{\leq}{\leqslant}
\renewcommand{\geq}{\geqslant}
\renewcommand{\ge}{\geqslant}

\author{Reza Mohammadpour\orcidA{}}
\address{Reza Mohammadpour, Department of Mathematics, Uppsala University,
  Box 480, SE-751 06 Uppsala, Sweden}
\email{reza.mohammadpour(a)math.uu.se}

\author{Anthony Quas}
\address{Anthony Quas, Department of Mathematics \& Statistics,
  University of Victoria, Victoria BC, Canada V8W 2Y2}
\email{aquas(a)uvic.ca}
\thanks{Corresponding Author: Anthony Quas}

\date{\today}

\subjclass[2020]{37H15,  37D35. }

\keywords{Thermodynamic formalism, matrix cocycles, equilibrium measures, freezing phase transition.}%
\email{}

\begin{document}
\title[Non-unique equilibrium measures and freezing phase transition]{Non-unique equilibrium measures
and freezing phase transitions for matrix cocycles for negative $t$}
\begin{abstract}
We consider the one-step matrix cocycle generated by a particular pair of non-negative parabolic matrices 
and study the equilibrium measures for $t\log \|\A\|$ as $t$ runs over the reals. We show 
that there is a freezing first order phase transition at $t=-2$ so that for $t\le -2$ the 
equilibrium measure is non-unique and supported on the two fixed points, while for $t>-2$, the 
equilibrium measure is unique, non-atomic and fully supported. The phase transition 
closely resembles the classical Hofbauer example. In particular, our example shows that there may 
be non-unique equilibrium measures for negative $t$ even if the cocycle is strongly irreducible 
and proximal. 
\end{abstract}
\maketitle

\section{Introduction and statement of results}
In this paper, we study the thermodynamic formalism for matrix cocycles. We will show the 
existence of equilibrium measures of the logarithm of the $t$th power of the
norm of a particular matrix cocycle (that satisfies the strong irreducibility
and proximality conditions) for all $t \in \mathbb{R}$.

We say that $(X, T)$ is a topological dynamical system if $X$ is a compact metric 
space and $T$ is a continuous map from $X$ to $X$. We say that 
$\Phi:=(\log \phi_{n})_{n=1}^{\infty}$  is a \emph{sub-additive potential} 
over $(X, T)$ if each $\phi_{n}$ is a continuous positive-valued function on $X$ such that
\[
\phi_{n+m}(x) \leq \phi_{n}(x) \phi_{m}(T^{n}(x)) \quad \forall x\in X, m,n \in \N.
\]
Similarly, we call a sequence of continuous functions (potentials) $\Phi=(\log \phi_n)_{n \in \N}$
\emph{super-additive} if
$$
 \phi_{n}(x) \phi_{m}(T^{n}(x)) \le \phi_{n+m}(x) \quad \forall x\in X, m,n \in \N.
$$
A potential $\Phi=(\log\phi_n)_{n \in \N}$ is \emph{almost additive} if there is a $C>0$ such that for all $x\in X$
and all $m,n\in\N$
$$
\frac 1C\le \frac{\phi_{n+m}(x)}{\phi_n(x)\phi_m(T^n(x))}\le C.
$$

Given a non-additive potential $\Phi$, an \emph{equilibrium measure} is a $T$-invariant
probability measure $\mu$ for which $p(\mu)=\sup_{\nu\in \M(T)}p(\nu)$
where $p(\nu)=h_{\nu}(T)+\lim_{n\to\infty}\frac 1n\int \log\phi_n\,d\nu$ and
$\M(T)$ denotes the collection of invariant probability measures. 

For sub-additive potentials over subshifts,  existence of equilibrium measures follows 
from upper semi-continuity (e.g., \cite{Barreira-eq, Falconer, CFH08}):
both $\mu\mapsto h_{\mu}(T)$ and $\mu\mapsto\lim_{n\to\infty}\frac 1n\int\log\phi_n\,d\mu$ are upper
semi-continuous and the existence of measures maximizing $h_{\mu}(T)+\lim_{n\to \infty}\frac 1n\int\log\phi_n\,d\mu$
follows from weak$^*$-compactness of the space of invariant probability measures. The super-additive case 
is more delicate because the entropy is upper semi-continuous, while the limit of the integrals is 
lower semi-continuous.

A \emph{matrix cocycle} $\A$ over a topological dynamical system $(X, T)$ is generated by
a continuous map 
$\A \colon X \to \glr$. For $n\in \N$ and $x \in X$, we define the product of $\A$ over
the orbit segment of length $n$ as
\[
\A^n(x):= \A(T^{n-1}(x)) \ldots \A(x).\]

A well-studied class of matrix cocycles are \emph{one-step cocycles} which are defined 
as follows. Assume that $\Si=\{1,...,k\}^{\Z}$ is a symbolic space and 
$T:\Si \rightarrow \Si$  is the shift map, i.e. $T(x_{l})_{l\in \Z}=(x_{l+1})_{l\in \Z}$. 
Given a $k$-tuple of matrices $\textbf{A}=(A_{1},\ldots,A_{k})\in \glr^{k}$ , 
we associate with it the locally constant map $\mathcal{A}:\Si \rightarrow \glr$ 
given by $\mathcal{A}(x)=A_{x_{0}}$.
The $k$-tuple of 
matrices $\textbf{A}$ is called the generator of the one-step cocycle $\A$. For any 
length $n$ word $I=i_{0}, \ldots, i_{n-1},$  we denote 
\[
\mathcal{A}_{I}:=A_{i_{n-1}}\ldots A_{i_{0}}.
\]
Therefore, when $\A$ is a one-step cocycle, 
\[ 
\A^n(x)=\A_{x|_{[0,n)}}=A_{x_{n-1}}\ldots A_{x_{0}}.
\] 
In this paper, we focus on the norm potential of matrix cocycles, which provide 
well-known examples of non-additive potentials.
If $\mathcal{A} : \Sigma \to \mathrm{GL}_d(\mathbb{R})$ is a matrix cocycle and $t\in \R$, 
then $t\Phi_{\A}:=(t\log \|\A^{n}\|)_{n=1}^{\infty}$ is sub-additive when $t \geq 0$ and 
super-additive when $t<0$. By the results mentioned above, when $t\geq 0$, there is an 
equilibrium measure for $t\Phi_{\A}$.
It is known that if a matrix cocycle $\A$ satisfies the quasi-multiplicativity property, 
then there is a unique equilibrium measure with the Gibbs property for $t\Phi_{\A}$ for all $t\in\R_+$
(see e.g., \cite{feng09, feng11,  park20, MP-uniqueequ}).

In the super-additive case, $t\Phi_{\mathcal{A}}$ for $t < 0$, much less is known. Apart from some 
well-understood cases, such as the strongly conformal, reducible, or dominated 
settings (see e.g., \cite[Proposition 5.8]{Moh22-Lyapunov}), there are
not many general results concerning 
equilibrium measures for $t\Phi_{\mathcal{A}}$ in the super-additive regime. An exception is the recent results in \cite{Rush, Mohammadpour-Varandas-statistical}, which apply to values of $t$ in a neighborhood of zero. 

The following theorem, our main result, gives a complete picture of the matrix equilibrium measures for $(t\Phi_\A)$
for a particular one-step matrix-valued cocycle for all $t\in\R$. We show that there is a phase transition at an
explicit value, $t=-2$.

\begin{thm}\label{t:mainresult1}
Let $\A:\{0,1\}^{\Z} \to \gl2$ be a one-step cocycle generated by 
\[
A_0 = \begin{bmatrix}
1 & 0 \\
1 & 1
\end{bmatrix}, \quad
A_1 = \begin{bmatrix}
1 & 1 \\
0 & 1
\end{bmatrix}.
\] 
Then the following hold:
\begin{enumerate}[(i)]
    \item $t\mapsto P(t\Phi_\A)$ is non-decreasing and convex  on $\R$;\label{it:nondec}
    \item For $t\le -2$, $P(t\Phi_\A)=0$ and the equilibrium measures for $t\Phi_\A$ are
    precisely $\delta_{\bar 0}$ and $\delta_{\bar 1}$;\label{it:precritical}
    \item For $t>-2$, $P(t\Phi_\A)>0$ and there is a unique equilibrium measure, $\mu_t$ for $t\Phi_\A$. The
    measure $\mu_t$ is fully supported on $\{0,1\}^\Z$.\label{it:postcritical}
\end{enumerate}
\end{thm}
We recall that matrix  $A \in \gl2$ is \textit{proximal} if it has two real 
eigenvalues with unequal absolute values. Assume that $(A_1, \ldots, A_{\ell})\in \gl2^{\ell}$ 
generate a one-step cocycle $\A \colon \Si \to \gl2$. We say that $\A$ is proximal 
if the semigroup generated by $\{A_1, \ldots, A_{\ell}\}$  contains a proximal element; 
that is, if there exists a finite product of the matrices $A_1, \ldots, A_{\ell}$ that 
is proximal. We also say that $\A \colon \Si \to \gl2$ is \textit{strongly irreducible} 
if there does not exist a finite collection $V_{1}, \ldots, V_{m}$ of non-zero proper 
subspaces $V_{j}$ such that 
$A_{i}\left(\bigcup_{j=1}^{m} V_{j}\right)=
\bigcup_{j=1}^{m} V_{j}$ for every $i=1, \ldots, \ell$. 
Showing that
$\mathcal A$ above is strongly irreducible and proximal, will give the following.

\begin{cor}\label{cor}
There exists a strongly irreducible and proximal one-step cocycle for which the 
equilibrium measure is for $t\Phi_{\A}$ is not unique for some $t<0$.
 \end{cor}

Theorem \ref{t:mainresult1} provides a counterpart to \cite[Theorem 1.1]{Rush}, where 
it is shown in generality that there is a unique equilibrium measure for the potential $t\Phi_{\A}$ for all 
$t$ in some neighborhood of zero. Corollary \ref{cor} should be compared to \cite[Proposition 10.3]{Rush}, 
where an example of a one-step cocycle is given for which there does not exist an 
equilibrium measure satisfying the Gibbs property.

In the classical additive thermodynamic formalism, equilibrium measures are the measures
for which $h_{\mu}(T)+\beta\int\phi\,d\mu$ achieves its maximum. The parameter $\beta$ is often 
referred to as the inverse temperature. If the underlying dynamical system
is a full shift and $\phi$ is H\"older continuous, the pressure is an analytic function of $\beta$
that is strictly convex except for the case where $\phi$ is cohomologous
to a constant (see \cite{Ruelle-book}). This implies that the equilibrium measures 
are distinct for distinct values of $\beta$. Invariant measures for which $\int\phi\,d\mu$
achieves its maximal value are known as maximizing measures. 
The term \emph{freezing phase transition} refers to the situation where the 
equilibrium measures for all inverse temperatures $\beta>\beta_c$ agree with a maximizing measure. 
From the above description, this can never occur for H\"older continuous potentials
\cite{ChazKuchQuas}. On the other hand, a well-known example of a continuous potential
that exhibits a freezing phase transition was was constructed by Hofbauer \cite{hofbauer1977nonuniqueness} 
(see also Ledrappier \cite{Ledrappier-phase} for a simplified proof). 
Although the example in this paper deals 
with non-additive matrix norm potentials rather than additive potentials, there is a strong
parallel with the Hofbauer example. 
The proof of the existence of the phase transition is elementary and self-contained. 

A related phenomenon occurs in the paper of Rush \cite{Rush}, where an example is given
of a one-step cocycle $\A_R$ consisting of an irrational rotation and a hyperbolic matrix. 
For that example, it was shown that $P(t\Phi_{\A_R})$ is constant on an interval $(-\infty,t_c]$
and strictly greater for all $t>t_c$. Rush's proof relies on multifractal formalism
computations \cite{DGR19} and does not give a construction of the equilibrium measures. 

\begin{figure}[ht]
\centering
\begin{tikzpicture}[scale=1.2]
  \draw[->] (-2.5,0) -- (2,0) node[right] {$t$};
  \draw[->] (0,-1) -- (0,3) node[above] {$P(t\Phi_{\mathcal{A}})$};

  \draw[blue, thick, domain=-2.5:-1.2, smooth, samples=100]
    plot(\x, {0.2*\x*\x + 1*\x + 1.25});

  \draw[blue, thick, domain=-1.2:-0.5, smooth, samples=100]
    plot(\x, {0.2*\x*\x + 1*\x + 1.25});

  \draw[red, thick, domain=-0.5:0, smooth, samples=100]
    plot(\x, {0.2*\x*\x + 1*\x + 1.25});

  \draw[black, thick, domain=0:1, smooth, samples=100]
    plot(\x, {0.2*\x*\x + 1*\x + 1.25});

  \draw[black, thick] (-2.5,0.1) -- (-2.5,-0.1) node[below left] {$t_c$};
   \draw[ blue, thick] (-2.5,0) -- (-4,0);


 

 
\end{tikzpicture}

\caption{We give a complete picture of the pressure for the
matrix cocycle that we study.
For $t> -2$, there is an equilibrium measure supported off the fixed points, and for $t\leq -2$, 
there are equilibrium measures supported at the fixed points. 
These are the only ergodic equilibrium measures. 
Note that for $t\geq 0$, the description of the equilibrium measure follows from 
Feng \cite{feng09, feng11} (in black) and for $t$ close to zero, the description follows from Rush \cite{Rush} 
(in red).}
\end{figure}
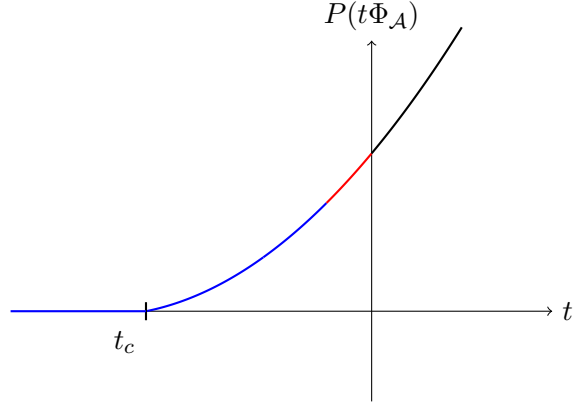


\subsection{Acknowledgements.}
Reza Mohammadpour is supported by the Swedish Research Council grant 104651320 and the Knut and Alice Wallenberg Foundation.
Anthony Quas is supported by an NSERC Discovery Grant. We thank the referees for their helpful suggestions.
\section{preliminaries}
\subsection{Set-up}\label{set up}

For each $n\in\N$, we define $\Sig_{n}$ to be the set of all length $n$ words of $\Si$, 
and we define $\Sig_{\ast}:=\bigcup_{n\in \N}\Sig_{n}$ to be the set of all words. 
For $m<0\le n$ and any sequence $a_m,\ldots,a_n$, we denote the \emph{cylinder set}
$\{x\colon x_i=a_i\text{ for $m\le i\le n$}\}$ by $[a_m\ldots a_{-1}.a_0\ldots a_n]$. 

The shift space $\Si$ is compact in the topology
generated by the cylinder sets. Moreover, the cylinder sets are open and closed in
this topology and they generate the Borel $\sigma$-algebra $\B$. 

 \subsection{Non-additive thermodynamic formalism}
Assume $\left(A_1, \ldots, A_k\right) \in GL(d, \R)^{k}$ generate a one-step cocycle 
$\A:\Sig \to GL(d, \R)$. For $t \in \R$, the {\em topological pressure} of $t\Phi_{\A}$ is defined by
\[ P(t\Phi_{\A}):=\lim_{n \to \infty}\tfrac{1}{n} \log s_n(t) , \]
where $s_{n}(t):=\sum_{I \in \Sig_n} \|\A_{I}\|^{t}$. Note that the existence 
of the limit follows from the sub-multiplicativity of $\|\cdot\|$.


Let $ \mu \in \M(T)$. We define the first Lyapunov exponent of $\A$ with respect
to $\mu$ and $T$ to be
\[\chi_1(\mu, \A):=\lim_{n\to \infty}\frac{1}{n} \int \log \|\A^{n}(x)\| d\mu(x),\]
where $\| \cdot \|$ denotes the operator norm. 
For simplicity, we denote $\chi(\mu, \A):=\chi_1(\mu, \A)$.

We recall that the \textit{Kolmogorov-Sinai entropy} of $\mu$ with respect to $T$ is
$$
h_{\mu}(T):=\lim _{n \rightarrow \infty} \frac{1}{n} \sum_{I \in \Sigma_n} \mu([\I]) \log \mu([\I]).
$$

Cao, Feng and Huang \cite{CFH08} proved a variational principle formula for the topological 
pressure of sub-additive potentials, while the counterpart for super-additive potentials was 
established by Cao, Pesin and Zhao \cite{CPZ}. 
More recently,  \cite{Mohammadpour-Varandas, Moh23} proved a variational principle for the 
generalized singular value function, which is a generalization of the family of potentials $\Phi_\A$ 
and is neither sub-additive nor super-additive (we refer the reader to 
\cite[Theorem B]{Mohammadpour-Varandas} for more details).
Hence, for any $t \in \R$,
\begin{equation}\label{varitional}
P(t\Phi_{\A})=\sup \bigg\{h_{\mu}(T)+t\chi(\mu, \A): \mu \in \mathcal{M}_{\text{inv}}(T) \bigg\}.
\end{equation}

\medskip
Any invariant measure $\mu \in \mathcal{M}_{\text{inv}}(T)$ achieving the
supremum in \eqref{varitional} is called an \textit{equilibrium measure} of $t\Phi_{\A}$. 
In other words, we say that $\mu_t$ is an {\em equilibrium measure} for $t\Phi_{\A}$ if 
\begin{equation}\label{equilibrium measure}
    P(t\Phi_{\A})= h_{\mu_t}(T)+t\chi(\mu_t, \A).
\end{equation}

We say that a probability measure $\mu_t \in \M(T)$ is a {\em Gibbs measure} for $t\Phi_{\A}$ 
if there exist $C_1, C_2>0$ such that for any $n \in \N$ and $I \in \Sigma_n$
\[
   C_1 \leq \frac{\mu_t\left([I]\right)}{e^{-nP(t\Phi_{\A}) }\|\A_{I}\|^t} \leq C_2.  
\]

\subsection{Induced maps}\label{subsec:induced}
We recall some definitions and fundamental properties of Kakutani towers.

Let $T$ be an ergodic, invertible, measure-preserving transformation on the probability space 
$(\Omega, \mathcal{B}, \mu)$, and let $D \subset \Omega$ be a measurable set with 
positive measure. For each $x \in D$, define the return time function $r_1(x):=r_D(x) = 
\inf \{ n > 0 : T^n(x) \in D \}$ and $r_k(x):=r_{k-1}(x)+r_1(T^{r_{k-1}(x)}(x))$ for each 
$k \in \mathbb{N}$. Also, for each $k\in \N$ let $D_k = \{ x \in D : r_D(x) = k \}$. 
The collection $\mathcal P_D=\{D_k\colon k\in\N\}$ forms a partition of $D$.
The induced transformation $T_D$ on the space $(D, \mathcal{B}_D, \mu_D)$ is given by 
$T_D(x) = T^{r_D(x)}(x)$. This map preserves the induced measure $\mu_D$, defined by 
$\mu_D(B) = \mu(D \cap B)/\mu(D)$, and the $\sigma$-algebra $\mathcal{B}_D$ 
consists of all sets of the form $B \cap D$, with $B \in \mathcal{B}$.

We will use the following three properties:
\begin{enumerate}
    \item The collection $\mathcal{P} = 
    \{ T^i D_k : k \in \mathbb{N},\ 0 < i < k \}$ forms a partition of $\Omega$;

    \item The measure $\mu_D$ is $T_D$-invariant and ergodic;
    
    \item The $\sigma$-algebra generated by $\mathcal{P_D}$ under the map $T_D$ 
    coincides with the restriction to $D$ of the $\sigma$-algebra generated by the collection 
    $\mathcal{Q} = \{ D, D^c \}$ under $T$.
\end{enumerate}

Let $\A:\Sigma \to \glr$ be a one-step cocycle and $D\subset \Sigma$ be as above.
We denote $\A_{D}(x)=\A^{r_1(x)}(x)$. We state the following facts that we use a number of times.
\begin{equation}\label{AbramovKnill}
\begin{aligned}
    h_{\mu}(T)&=\mu(D)h_{\mu_D}(T_{D})\text{; and}\\
    \chi(\mu,\A)&=\mu(D)\chi(\mu_D,\A_D).
\end{aligned}
\end{equation}
The first is Abramov's formula, and the second is due to Knill \cite{Knill1991}.

We recall that if $\nu_D$ is a $T_D$-invariant measure on $D$, there is a corresponding
$T$-invariant measure $\nu$ on $\Sigma$, called the \emph{lift} of $\nu_D$ to $\Sigma$. In the case
where $\int r_D\,d\nu_D$ is finite, the measure $\nu$ is a probability measure and $\nu_D$ satisfies
$\nu_D(A)=\nu(A\cap D)/\nu(D)$ as above. In the case where $\int r_D\,d\nu_D$ is infinite, the 
measure $\nu$ is a $\sigma$-finite invariant measure. See \cite[\S1.5]{Aaronson} for more details.

\section{Preliminary lemmas and proofs}
For the remainder of the article, we fix $ \Sigma = \{0,1\}^{\mathbb{Z}} $, 
the full shift $ T: \Sigma \to \Sigma $, 
and the locally constant map $ \A: \Sigma \to \gl2 $ generated by
\[
A_0 = \begin{bmatrix}
1 & 0 \\
1 & 1
\end{bmatrix}, \quad
A_1 = \begin{bmatrix}
1 & 1 \\
0 & 1
\end{bmatrix}.
\]

\begin{lem}\label{lower bound}
Let $I = \bigl(0^{i_0} 1^{j_0} 0^{i_1}1^{j_1} \dots 0^{i_{k-1}} 1^{j_{k-1}}\bigr), $ where $i_\ell, j_\ell\in \mathbb{N}$. Then,
\[
\|\mathcal{A}_I\| \;\ge\; (1+i_0j_0)\ldots (1+i_{k-1}j_{k-1}).
\]
\end{lem}
\begin{proof}
Note that
\[
A_0^{i} = \begin{bmatrix}1 & 0 \\ i & 1\end{bmatrix}, \qquad
A_1^{j} = \begin{bmatrix}1 & j \\ 0 & 1\end{bmatrix}.
\]
Therefore, defining $B_l:=A_1^{j_l}A_0^{i_l}$, we have $B_l=\begin{bmatrix}1+i_lj_l & j_l \\ i_l & 1\end{bmatrix}$.

One may show by induction that $(B_{k-1}\ldots B_0)_{11}\geq (1+i_0j_0)\ldots (1+i_{k-1}j_{k-1})$, 
where $B_{11}$ refers
to the upper left entry of $B \in \gl2$. 
Hence, $\|B_{k-1}\ldots B_0\| \geq (1+i_0j_0)\ldots (1+i_{k-1}j_{k-1})$. 
\end{proof}

\begin{lem}\label{positive LE}
For any ergodic measure $\mu$ with $\mu\ne \delta_{\bar0},\delta_{\bar 1}$,
$\chi(\mu, \A)>0$.
\end{lem}
\begin{proof}
Let $D=[1.0]=\{x\colon x_{-1}=1, x_0=0\}$. Since $\mu$ is not supported on either fixed
point, $\mu([0])>0$ and $\mu([1])>0$. We have $[0]=[0.0]\cup[1.0]$. If $\mu([0]\setminus[0.0])$ were equal to 0,
then $[0]$ would be an invariant set. So, by ergodicity, we see $\mu(D)>0$. We define a partition of $D$,
$\{X_{i,j}\colon i,j\in \N\}$, where
$$
X_{i,j}=[1.0^i1^j0].
$$
We induce on $D$ and let $T_D$ be the return map with ergodic induced measure $\mu_D$.
Applying Lemma \ref{lower bound}, we see if $x\in X_{i_0,j_0}\cap T_D^{-1}X_{i_1,j_1}\cap\ldots\cap
T_D^{-(n-1)}X_{i_{n-1},j_{n-1}}$,
$$
\log\|\A_D^n(x)\|\ge \sum_{k=0}^{n-1} \log(1+i_kj_k).
$$
 Defining $f(x)$ by $f(x)=\log (1+ij)$ if $x\in X_{i,j}$, this can be rewritten as
$$
\log\|\A_D^n(x)\|\ge \sum_{k=0}^{n-1}f(T_D^k(x)).
$$
Since $f\ge\log 2$, we see $\chi(\mu_D, \A_D)\geq \log 2$, so that by \eqref{AbramovKnill},
$\chi(\mu,\A)>0$ as
required.
\end{proof}

For the proof of the main theorem, we will need to study the action of matrices with non-negative entries on the
projective non-negative quadrant. Given a non-zero vector $\mathbf v$
with non-negative entries,
the ray with direction $\mathbf v$ is $[\mathbf v]_+=\{s\mathbf v\colon s\in \R^+\}$. 
The space of non-negative rays is identified with $[0,1]$ by the correspondence
$x\mapsto \left[\begin{smallmatrix}x\\1-x\end{smallmatrix}\right]_+$.
Given a matrix $A$ with non-negative entries, $A$ acts on the space of rays by
$A[\mathbf v]_+=[A\mathbf v]_+$.

In terms of the interval parameterization, the action
of the matrix $A=\left[\begin{smallmatrix}a&b\\c&d\end{smallmatrix}\right]$ is given by
the M\"obius transformation
$$
f_A(x)=\frac{ax+b(1-x)}{(a+c)x+(b+d)(1-x)}.
$$
To see this, notice that
$$
\begin{aligned}
\begin{bmatrix}
    a&b\\c&d
\end{bmatrix}\begin{bmatrix}
    x\\1-x
\end{bmatrix}
&=\begin{bmatrix}
    ax+b(1-x)\\cx+d(1-x)
\end{bmatrix}\\
&=\big((a+c)x+(b+d)(1-x)\big)\begin{bmatrix}
    \frac{ax+b(1-x)}{(a+c)x+(b+d)(1-x)}\\
    \frac{cx+d(1-x)}{(a+c)x+(b+d)(1-x)}.
\end{bmatrix}
\end{aligned}
$$

Hence, we define the action on $[0,1]$ by 
$$
A\begin{bmatrix}
    x\\1-x
\end{bmatrix}_+=\begin{bmatrix}
    f_A(x)\\1-f_A(x)
\end{bmatrix}_+.
$$

Note that the non-standard form of the M\"obius transformation arises as we are
parameterizing projective space by the diagonal line $x+y=1$ (in which the parameter range 
for the non-negative projective space is
$[0,1]$) rather than the line $x=1$ (where the parameter range would be $[0,\infty]$).
Since $AB[\mathbf v]_+=A[B\mathbf v]_+$, one can see that $f_{AB}=f_A\circ f_B$. 
The same M\"obius transformations appear in \cite{JenkinsonPollicott}.
The following lemma shows how the behaviour of these compositions of M\"obius transformations 
is related to the norm of the products of matrices that we wish to study.

\begin{lem}\label{lem:Mobius}
    Let $f_0$ and $f_1$ denote the mappings $f_0(x)=\frac{x}{1+x}$ and $f_1(x)=\frac{1}{2-x}$,
    the M\"obius transformations of $[0,1]$ associated with $A_0=\left[\begin{smallmatrix}1&0\\1&1\end{smallmatrix}\right]$
    and $A_1=\left[\begin{smallmatrix}1&1\\0&1\end{smallmatrix}\right]$.
   For any $n \in \N$ and any word $I=i_0\ldots i_{n-1}\in\Sigma_n$, let $f_I=f_{i_{n-1}}\circ
    \ldots\circ f_{i_{0}}$ be the M\"obius transformation associated with $\A_I$
    and let $J_I=f_I([0,1])$. 
    Then for each $n\in\N$, 
    \begin{enumerate}[(i)]
        \item $\bigcup_{I\in\Sigma_n}J_I=[0,1]$; \label{it:cover}
        \item the sets $\operatorname{int}(J_I)$ are pairwise disjoint as $I$ runs over $\Sigma_n$;\label{it:disjoint}
        \item there exists $C>1$ such that for all $n\in\N$ and all $I\in\Sigma_n$, \label{it:normest}
        \begin{equation}\label{eq:normA2bound}
        \frac 1C\,\frac{\operatorname{Leb}(J_I)}{k(I)}\le \frac 1{\|\A_I\|^2}\le C\frac{\operatorname{Leb}(J_I)}{k(I)},
        \end{equation}
        where $1\le k(I)\le n$ is the largest $k$ such that $I$ starts with $0^k$ or $1^k$. 
    \end{enumerate}
\end{lem}

\begin{proof}
One can see that $f_0$ is strictly increasing and maps $[0,1]$ to $[0,\frac12]$, 
while $f_1$ is strictly increasing and maps $[0,1]$ to $[\frac12,1]$, so that claims \ref{it:cover} and \ref{it:disjoint} hold for $n=1$. 
Given that the claim holds for $n$, we have $J_{I0}=f_0(J_I)$ so that the $J_{I0}$ with $I$ running over $\Sigma_n$
cover $J_0$ and have pairwise disjoint interiors. Similarly the $J_{I1}$ cover
$J_1$ and have pairwise disjoint interiors. Since $J_0$ and $J_1$ have disjoint 
interiors, both claims hold for $n+1$.


We now establish claim \ref{it:normest}. 
We use the notation 
$X_I=\Theta(Y_I)$ to mean that the ratio of the quantities is uniformly bounded 
above and below independently of $n$ and $I$.
For any $I$, $\A_I$ maps $[0,1]^2$ to a parallelogram
of area 1 with side lengths $\|\A_Ie_1\|$ and $\|\A_Ie_2\|$, so that $\|\A_Ie_1\|\|\A_Ie_2\|\sin\theta_I=1$,
where $\theta_I$ is the angle of the parallelogram at the origin.
Since $0<\theta_I<\frac\pi 2$, we have $\frac 2\pi \theta_I<\sin\theta_I<\theta_I$
so that $\theta_I=\Theta(\sin\theta_I)$
and 
$$
\theta_I=\Theta\left(\frac{1}{\|\A_Ie_1\|\|\A_Ie_2\|}\right).
$$
Since the segment of the parallelogram
on the line $x+y=1$ is of length $\sqrt 2\operatorname{Leb}(J_I)$, 
we see $\operatorname{Leb}(J_I)=\Theta(\theta_I)$.
Hence we have established 
$$
\operatorname{Leb}(J_I)=\Theta\left(\frac{1}{\|\A_Ie_1\|\|\A_Ie_2\|}\right).
$$
To finish the proof of the claim, we need to show that 
\begin{equation}\label{eq:tocheck}
\|\A_I\|^2=\Theta\Big(k(I)\|\A_Ie_1\|\|\A_Ie_2\|\Big).
\end{equation}

In the case that $I=0^n$ or $1^n$, $\A_I$
is $\left[\begin{smallmatrix}1&0\\n&1\end{smallmatrix}\right]$ 
or $\left[\begin{smallmatrix}1&n\\0&1\end{smallmatrix}\right]$ 
respectively. We see that $k(I)=n$, $\|\A_I\|^2=\Theta(n^2)$, and 
$\|\A_Ie_1\|\|\A_Ie_2\|=\Theta(n)$, so \eqref{eq:tocheck} holds in this case. 
Otherwise, let $I=0^k1\tilde I$ or $1^k0\tilde I$. We deal with the first case, but the second
is exactly similar. Let $B=\A_{\tilde I}$ so that $\A_I=BA_1A_0^{k}$.

For a non-negative matrix, we have the bound
$$
\|A\|\le \left\|A\left[\begin{smallmatrix}1\\1\end{smallmatrix}\right]\right\|\le \sqrt 2\|A\|.
$$
so that $\|A\|=\Theta(\|A\left[\begin{smallmatrix}1\\1\end{smallmatrix}\right]\|)$.

We now have 
\begin{equation}\label{eq:ingr1}
\begin{aligned}
    \|\A_I\|&=\Theta(\|BA_1A_0^k\left[\begin{smallmatrix}1\\1\end{smallmatrix}\right]\|)\\
    &=\Theta(\|B\left[\begin{smallmatrix}k+2\\k+1\end{smallmatrix}\right]\|)\\
    &=\Theta(k\|B\|).
\end{aligned}
\end{equation}
On the other hand, we have
\begin{equation}\label{eq:ingr2}
\begin{aligned}
    \|\A_Ie_1\|&=\|BA_1A_0^ke_1\|\\
    &=\|B\left[\begin{smallmatrix}k+1\\k\end{smallmatrix}\right]\|=\Theta(k\|B\|)\text{, while}\\
    \|\A_Ie_2\|&=\|BA_1A_0^ke_2\|\\
    &=\|B\left[\begin{smallmatrix}1\\1\end{smallmatrix}\right]\|=\Theta(\|B\|).
\end{aligned}
\end{equation}
Together, \eqref{eq:ingr1} and \eqref{eq:ingr2} establish \eqref{eq:tocheck}, completing the proof
of conclusion \ref{it:normest} of the theorem.
\end{proof}

For any $y\in [0,1]$, $y$ can be uniquely represented as $f_{i}(x)$ for some $i\in\{0,1\}$ and some $x\in[0,1]$
unless $y=\frac 12$
in which case $y$ can be written as $f_0(1)$ or $f_1(0)$. 
It follows that for each $n\in\N$, all but countably many $y\in[0,1]$ can be uniquely represented in the form $f_I(x)$
with $I\in\Sigma_n$ and $x\in[0,1]$.
We use the following map to iterate these preimages:
$$
S(x)=
\begin{cases}f_0^{-1}(x)=\frac x{1-x}&\text{if $x\in [0,\frac12)$;}\\
f_1^{-1}(x)=2-\frac 1x&\text{if $x\in [\frac 12,1]$.}
\end{cases}
$$
We also define 
$$
I(x)=
\begin{cases}0&\text{if $x\in [0,\frac12)$;}\\
1&\text{if $x\in [\frac 12,1]$,}
\end{cases}
$$
and $I_n:[0,1]\to\Sigma_n$ by 
$I_n(y):=I(S^{n-1}y),I(S^{n-2}(y))\ldots,I(y)$, so that $y\in f_{I_n(y)}([0,1])$.
We show that the map $S$ preserves an (infinite) ergodic absolutely continuous measure, so that we can
use ergodic theorems to control the behaviour of the quantities appearing in \eqref{eq:normA2bound}.

\begin{lem}\label{lem:infmeas}
Let $S$ and $I_n$ be as above. 
\begin{enumerate}[(i)]
\item For any $y\in[0,1]$ and any $n\in\N$, $y=f_{I_n(y)}(S^ny)$;
\item $S$ preserves an ergodic absolutely continuous $\sigma$-finite measure $\lambda$ with density
with respect to Lebesgue measure given by
$$
\rho(x)=\tfrac 1x+\tfrac1{1-x}.
$$
\end{enumerate}
\end{lem}

\begin{proof}
    The first part is straightforward as the sequence of $f_0$'s and $f_1$'s in $f_{I_n(y)}$
    invert the inverse images defining $S^n(y)$ one-by-one. 

    For the second part, we verify that the density of the push-forward of $\lambda$ under $S$ is
    $\rho$. The map $S$ is two-to-one: each $x\in[0,1]$ is the image under $S$ of $f_0(x)$ and
    of $f_1(x)$. We therefore have to verify that
    $$
    \rho(x)=\rho(f_0(x))f_0'(x)+\rho(f_1(x))f_1'(x).
    $$
This is a straightforward calculation.
The ergodicity of $(S,\lambda)$ follows from \cite{Thaler}.
\end{proof}

\begin{proof}[Proof of Corollary \ref{cor}]
We show that the matrix cocycle $\mathcal A$ in Theorem \ref{t:mainresult1}
is proximal and strongly irreducible. Proximality follows since
$A_0A_1=\left(\begin{smallmatrix}1&1\\1&2\end{smallmatrix}\right)$, which has
one eigenvalue inside and one outside the unit circle. 

Given any one-dimensional subspace $V$, if $V\ne \operatorname{lin}(e_2)$, then $A_0^n(V)$ is a
sequence of subspaces converging to $\operatorname{lin}(e_2)$. If $V\ne\operatorname{lin}(e_1)$, then $A_1^n(V)$ converges
to $\operatorname{lin}(e_1)$. It follows that there is no finite set of subspaces permuted by both $A_0$ and $A_1$. 
\end{proof}

In preparation for the proof of Theorem \ref{t:mainresult1}, we prove an
almost multiplicativity result. 

\begin{lem}[Almost multiplicativity]\label{Almost additivity}
Let $\C:=\big\{\left[\begin{smallmatrix}1+mn & n \\ m & 1\end{smallmatrix}\right] : 
m,n \in \N \big\}$. For each $n, m \in \N$, and any sequence $B_1,\ldots,B_{n+m}$ in $\C$, we have
\color{black}
\[
\tfrac{1}{2\sqrt{2}}\|B_{m+n} \ldots B_{m+1}\| \|B_{m}\ldots B_1\|
\leq 
\|B_{m+n}\ldots B_1\| 
\leq 
\|B_{m+n} \ldots B_{m+1}\|\|B_{m} \ldots B_{1}\|.
\]
\end{lem}

\begin{proof}
We prove the lemma in a series of claims. 
Let $P=\{(x, y): x\geq 0, y \geq 0\}$. Then we first claim 
\begin{equation}\label{eq:smallcone}
B_n\ldots B_1 P \subset \{(x,y): x\geq y \geq 0\}\text{ for any $B_1,\ldots B_n \in\C$}.
\end{equation}
To see this, since $B_1,\ldots,B_{n-1}$ are non-negative $B_{n-1}\ldots B_1P\subseteq P$. 
Then it is easy to check that $B_nP\subset\{(x,y)\colon x\geq y\ge 0\}$.

Next we claim 
\begin{equation}\label{eq:dom}
    B_n \ldots B_1 \e_1 \succeq B_n \ldots B_1 \e_2\text{ for any finite sequence 
of matrices in $\C$},
\end{equation}
where $(a, b)\succeq (c,d)$ means that $a\geq c$ and $b\geq d$.
To see this, since $B_1\e_1=(m_1n_1+1)\e_1+n_1\e_2$ and $B_1\e_2=m_1\e_1+1\e_2$, we see $B_1\e_1\succeq B_1\e_2$. 
If a matrix $B$ has non-negative entries, one can check $B\mathbf x\succeq B\mathbf y$ whenever
$\mathbf x\succeq \mathbf y$.

We next claim
\begin{equation}\label{eq:e1bound}
\|B_n \ldots B_1\e_1\|\ge \frac 1{\sqrt 2}\|B_n \ldots B_1\|
\text{ for any finite sequence of matrices in $\C$}.
\end{equation}

Since $(B_n\ldots B_1)^T(B_n\ldots B_1)$ has positive entries, by the Perron-Frobenius theorem,
the dominant eigenvector has positive entries. That is, there exists $\mathbf v$ with 
positive entries such that $\|\mathbf v\|=1$ and $\|B_n\ldots B_1\mathbf v\|=\|B_n\ldots B_1\|$.
Let $\mathbf v=\alpha\e_1+\beta\e_2$. Then by \eqref{eq:dom},
$B_n\ldots B_1((\alpha+\beta)\e_1)\succeq
B_n\ldots B_1\mathbf v$, so that taking norms, $(\alpha+\beta)\|B_n\ldots B_1\e_1\|
\ge \|B_n\ldots B_1\|$. Since $\|v\|=1$, we see $\alpha+\beta\le \sqrt 2$, so that
$\|B_n\ldots B_1\e_1\|\ge \frac 1{\sqrt 2}\|B_n\ldots B_1\|$ as required.    

We now complete the proof.
   Let $B_m\ldots B_1\e_1=\alpha e_1+\beta e_2$. By \eqref{eq:smallcone}, $\alpha\ge \beta$.
    Hence 
$$
\begin{aligned}
\alpha&=\|\alpha e_1\|\ge \tfrac 1{\sqrt 2}\|B_m\ldots B_1\e_1\|\\
&\ge\tfrac 12\|B_m\ldots B_1\|,
\end{aligned}
$$
where we used \eqref{eq:e1bound}. 
We then have $B_{m+n}\ldots B_1\e_1\succeq \alpha B_{m+n}\ldots B_{m+1}\e_1$,
so that 
$$
\begin{aligned}
\|B_{m+n}\ldots B_1\|&\ge \|B_{m+n}\ldots B_1\e_1\|\\
&\ge \alpha\|B_{m+n}\ldots B_{m+1}\e_1\|\\
&\ge \tfrac{\alpha}{\sqrt 2}\|B_{m+n}\ldots B_{m+1}\|.
\end{aligned}
$$
Substituting the earlier inequality for $\alpha$ establishes 
$$
\|B_{m+n}\ldots B_1\|\ge 
\tfrac{1}{2\sqrt{2}}\|B_{m+n} \ldots B_{m+1}\| \|B_{m}\ldots B_1\|.
$$ 
The other inequality follows from sub-multiplicativity. 
\end{proof}

The proof of Theorem \ref{t:mainresult1}\ref{it:postcritical} relies on the following theorem 
of Iommi and Yayama. The definitions of the BIP (big images and preimages) property and 
of Bowen sequences, in the hypotheses in the theorem, are Definitions 4.3 and 2.4
respectively in \cite{IommiYayama2012}.

\begin{thm}[{\cite[Proposition 3.1 and Theorem 4.1]{IommiYayama2012}}]\label{thm:IY}
Let $(\Omega,T)$ be a topologically mixing countable state Markov shift with the BIP
(big images and preimages)
property. Let $\Psi = (\log \psi_n)_{n\in\N}$ be an almost-additive 
Bowen sequence defined on $\Omega$. Then we have
\begin{enumerate}
    \item $P(\Psi)=\sup \left\{
    P(\Psi_Y):\text{ $Y$ is a Markov subshift of $\Omega$ with finitely many symbols}\right\}$;
    \label{it:VP}
    \item If $\sum_a\sup \psi_1|_{[a]} < \infty$
then there is a mixing Gibbs measure $\mu$ for $\Psi$. Moreover, if
$h_{\mu}(T) < \infty$, then $\mu$ is the unique equilibrium measure for $\Psi$.
\label{it:Gibbs}
\end{enumerate}

\end{thm}


In our context, $\Omega$ will be a countable symbol full shift (which automatically has the BIP property). 
For the potentials we consider, $\psi_n(\omega)$ only depends on $\omega_0,\ldots\omega_{n-1}$
which ensures that the Bowen property is satisfied. So to apply Theorem \ref{thm:IY}, it suffices
to check the almost additivity and summability conditions in the theorem.
In this case, in the first conclusion, we can consider systems $Y$ that are full subshifts
on finitely many symbols.

We restate Theorem \ref{t:mainresult1} for the convenience of the reader. 
\begin{thm*}
Let $\A:\{0,1\}^{\Z} \to \gl2$ be a one-step cocycle generated by 
\[
A_0 = \begin{bmatrix}
1 & 0 \\
1 & 1
\end{bmatrix}, \quad
A_1 = \begin{bmatrix}
1 & 1 \\
0 & 1
\end{bmatrix}.
\] 
Then the following hold:
\begin{enumerate}[(i)]
    \item $t\mapsto P(t\Phi_\A)$ is non-decreasing and convex  on $\R$;
    \item For $t\le -2$, $P(t\Phi_\A)=0$ and the equilibrium measures for $t\Phi_\A$ are
    precisely $\delta_{\bar 0}$ and $\delta_{\bar 1}$;
    \item For $t>-2$, $P(t\Phi_\A)>0$ and there is a unique equilibrium measure, $\mu_t$ for $t\Phi_\A$. The
    measure $\mu_t$ is fully supported on $\{0,1\}^\Z$.
\end{enumerate}
\end{thm*}
\color{black}

\begin{proof}[Proof of Theorem \ref{t:mainresult1}]
Since $\|\A^n(x)\|>1$ for each $x$ and each $n\in\N$, one can see that 
$t\mapsto P(t\Phi_\A)$ is non-decreasing in $t$. \cite[Lemma 2.2]{feng09} shows that $t\mapsto P(t\Phi_\A)$ is 
convex (and hence continuous). This establishes conclusion \ref{it:nondec} of the theorem.

Since $\chi(\delta_{\bar i},\A)=0$ and $h_{\delta_{\bar i}}(T)=0$ for $i=0,1$, by 
the variational principle, we see $P(t\Phi_\A)\ge 0$ for all $t\in\R$. 

We now show that $P(-2\Phi_\A)=0$ by showing that $\sum_{I\in\Sigma_n}\|\A_I\|^{-2}$ is uniformly bounded in $n$. 

Let $n\in\N$. By Lemma \ref{lem:Mobius}\ref{it:normest}, for each $I\in\Sigma_n$,
$1/\|\A_I\|^2\le C\operatorname{Leb}(J_I)/k(I)\le C\operatorname{Leb}(J_I)$,
where $C$ is independent of $I$ and $n$. Since the intervals $J_I$ have disjoint interiors and
contained in $[0,1]$, we see 
$$
\sum_{I\in \Sigma_n}1/\|\A_I\|^2\le C.
$$
Since this holds for all $n$, it follows that $P(-2\Phi_\A)\le 0$.
Hence,
$P(t\Phi_\A)=0$ for $t\in(-\infty,-2]$ and $\delta_{\bar 0}$ and $\delta_{\bar 1}$ are 
equilibrium measures for $t\Phi_\A$ for all $t\le-2$. 

We now show that if $t<-2$, $\delta_{\bar 0}$ and $\delta_{\bar 1}$ are the only ergodic equilibrium 
measures for $t\Phi_\A$. To see this, let $t<-2$ and suppose $\mu$ is an ergodic equilibrium measure
for $t\Phi_\A$. Then 
$$
\begin{aligned}
    0&=P(-2\Phi_\A)\\
    &\ge h_\mu(T)-2\chi(\mu,\A)\\
    &=h_\mu(T)+t\chi(\mu,\A)+(-2-t)\chi(\mu,\A)\\
    &=P(t\Phi_\A)+(-2-t)\chi(\mu,\A)\\
    &=(-2-t)\chi(\mu,\A).
\end{aligned}
$$
Hence $\chi(\mu,\A)\le 0$ so by Lemma \ref{positive LE}, $\mu$ is $\delta_{\bar 0}$ or $\delta_{\bar 1}$ as
claimed.


This establishes most of 
conclusion \ref{it:precritical} of the theorem.
All that remains is the claim that $\delta_{\bar 0}$
and $\delta_{\bar 1}$ are the only ergodic equilibrium measures for $t=-2$. We address this part of the claim later.

We now show that for $t>-2$, $P(t\Phi_\A)>0$. 
Since $t\mapsto P(t\Phi_\A)$ is non-decreasing,
it suffices to show that for $P(t\Phi_\A)>0$ for all $t\in (-2,0)$.
Notice that by sub-multiplicativity of $\|\cdot\|$, and hence super-multiplicativity of $\|\cdot\|^t$,
$$
\sum_{I\in \Sigma_{n+m}}\|\A_I\|^t\ge \sum_{I\in\Sigma_n}\|\A_I\|^t\sum_{J\in\Sigma_m}\|\A_J\|^t.
$$
Hence, to show that $P(t\Phi_\A)>0$, it suffices to show $\sum_{I\in\Sigma_n}\|\A_I\|^t>1$ for some $n\in\N$.
Let $t\in (-2,0)$, writing
\color{black}
$t=-2+\alpha$ and let $n\in\N$. 
Then 
$$
\begin{aligned}
    \sum_{I\in \Sigma_n}\|\A_I\|^t&=\sum_{I\in\Sigma_n}\frac {\|\A_I\|^\alpha}{\|\A_I\|^2}\\
    &\ge \frac 1C\sum_{I\in\Sigma_n}\frac{\operatorname{Leb}(J_I)}{k(I)}\|\A_I\|^\alpha\\
    &=\int_0^1 f_n(x)\,dx,
\end{aligned}
$$
where we used Lemma \ref{lem:Mobius}\ref{it:normest}, and where 
$$
f_n(x)=\frac{\|\A_{I_n(x)}\|^\alpha}{C\,k(I_n(x))}.
$$
Therefore, it suffices to show that $f_n(x)\to\infty$ as $n\to\infty$
for $\operatorname{Leb}$-a.e. $x$. 
In fact, we prove the stronger statement
\begin{equation}\label{eq:diverg}
    \frac{\log\|\A_{I_n(x)}\|}{\max(1,\log k(I_n(x)))}\to\infty\text{ for $\operatorname{Leb}$-a.e. $x$.}
\end{equation}

To show this, we look at returns of the dynamical system $S$ introduced in 
Lemma \ref{lem:infmeas} to the set $B=[\frac 13,\frac 23]=f_0([\frac 12,1])\cup f_1([0,\frac 12])$. 
These are exactly the places where $I(x)\ne I(S(x))$, that is the places where blocks of 0's
switch to 1's and vice versa. 
As in section \ref{subsec:induced}, we define $r^j_B(x)$
to be the $j$th return time to $B$ (or the $j$th entry time if $x\not\in B$).
Let $E_j(x)$ be the $j$th excursion time, $r_B(S^{r^{j-1}_B(x)}(x))$. 
Let $V_n(x)=\mathbf 1_B(x)+\ldots+\mathbf 1_B(S^{n-1}x)$ denote the number of visits to $B$
in the first $n$ steps. 
We observe that 
\begin{equation}\label{eq:kbound}
k(I_n(x))\le E_{V_n(x)}(x).
\end{equation}
Similarly, by a calculation very similar to the one in Lemma \ref{lower bound} (except where
we induce on a change from 0's to 1's or 1's to 0's rather than just on changes from 1's to 0's),
we have
\begin{equation}{\label{eq:normbound}}
\|\A_{I_n(x)}\|\ge E_1(x)\cdots E_{V_n(x)-1}(x).
\end{equation}
Letting $\lambda_B$ be the conditional probability 
measure $\lambda_B(U)=\lambda(B\cap U)/\lambda(U)$,
we have that $\lambda_B$ is invariant for the transformation $S_B$. 
We claim that $\log E_0(x)$ is integrable on $B$ with respect to $\lambda_B$. 
To see this, notice that for $x\in B$, $E_0(x)=k$ if $I(x)\ne I(S(x))$,
$I(S(x))=\ldots=I(S^k(x))$ and $I(S^k(x))\ne I(S^{k+1}(x))$. 
That is for $x\in B$, $E_0(x)=k$ if and only if $x\in f_0([\frac k{k+1},\frac{k+1}{k+2}])
\cup f_1([\frac{1}{k+1},\frac{1}{k+2}])$. Hence, $\lambda_B(\{x\colon E_0(x)=k\})=\Theta(\frac 1{k^2})$. 
The integrability of $\log E_0$ follows. 

By the ergodic theorem (applied to the ergodic probability-preserving transformation
$S_B$), 
$$
\frac 1m\sum_{i=0}^{m-1}\log E_0(S_B^i(x))\to\int\log E_0\,d\lambda_B
$$
for $\lambda_B$-a.e. $x\in B$. 
Since $\lambda$ is ergodic for $S$, the same holds for $\lambda$-a.e.
$x\in [0,1]$, and hence for $\operatorname{Leb}$-a.e. $x\in [0,1]$. 
A corollary of the ergodic theorem shows that 
$$
\frac{\log E_0(S_B^m(x))}m\to 0
$$
for $\lambda_B$-a.e. $x\in B$. This also applies for $\operatorname{Leb}$-a.e. $x\in [0,1]$. 

As a consequence of these statements together with \eqref{eq:kbound} and \eqref{eq:normbound}, we see that for
$\operatorname{Leb}$-a.e. $x$, 
$$
\begin{aligned}
   &\liminf_{n\to\infty} \frac{\log \|\A_{I_n(x)}\|}{V_n(x)}>0\text{ while }\\
   &\lim_{n\to\infty}\frac{\log \max(1,\log k(I_n(x)))}{V_n(x)}=0.
\end{aligned}
$$
The claim \eqref{eq:diverg} follows, so that $P(t\Phi_\A)>0$ as required.

It remains to show that for $t>-2$, there exists a unique equilibrium 
measure supported on $\Sigma \setminus \{\bar{0}, \bar{1}\}$, while for $t=-2$,
there are no ergodic equilibrium measures other than $\delta_{\bar 0}$ and $\delta_{\bar 1}$. 

Let $t>-2$ be fixed and let $P=P(t\Phi_{\A})>0$.
We let $\psi_{s,n}(x)=e^{-nP}\|\A^n(x)\|^{s+t}$
and define a family of potentials (as $s$ runs over $\R$) by
$\Psi_s=(\log\psi_{s,n})_{n\in\N}$. By construction,
$P(\Psi_s)=-P+P((t+s)\Phi_\A)$, so that $P(\Psi_s)$ is defined
for all $s\in\R$ and is a convex function of $s$. In particular,
$P(\Psi_0)=0$.
We also define a potential on the induced system. Let
$$
\psi_{D,s,k}(x)=\psi_{s,r_k(x)}(x)\text{ for $x\in D$}.
$$
where $D=[1.0]=\{x\colon x_{-1}=1, x_0=0\}$, and define the induced potential 
$\Psi_{D,s}=(\log\psi_{D,s,k})_{k\in\N}$, so that

$$
\begin{aligned}
\psi_{D,s,k}(x)&=e^{-P\,r_k(x)}\|B_{m_{k-1},n_{k-1}}\cdots B_{m_0,n_0}\|^{s+t}\text{ if $T_D^jx\in X_{m_j,n_j}$ for $j=0,\ldots,k-1$.}
\end{aligned}
$$
We introduce a symbolic system that is isomorphic to $(D,T_D)$. Let $\Omega=(\N^2)^\Z$ be equipped
with the shift map $T_\Omega$. 
The isomorphism is $\pi\colon D\to\Omega$, given by
$$
\pi(x)_j=(m,n)\text{ if $T_D^j(x)\in X_{m,n}$.}
$$
This map is a bijection from the set of points in $D$ that return to $D$ infinitely often in the past and
the future (a set of measure 1 with respect to any invariant probability measure supported on $D$)
to $\Omega$. 

We define the one-step 
matrix cocycle $\A_\Omega(\omega)$ by
$$
\A_\Omega(\omega)=B_{m,n}:=\begin{bmatrix}
    mn&n\\m&1
\end{bmatrix}\text{ if $\omega\in[(m,n)]$}
$$
over the shift map  $(\Omega, T_{\Omega})$.
This has the property that 
$$
\A_\Omega^k(\pi(x))=\A_D^k(x)\text{\quad for all $k\in\N$ and all $x\in D$.}
$$
We define a potential on $\Omega$ by $\Psi_{\Omega,s}=(\log\psi_{\Omega,s,k})_{k\in\N}$, with
$$
\psi_{\Omega,s,k}(\omega)=e^{-(M_k+N_k)P}\|\A_\Omega^k(\omega)\|^{s+t},
$$
where $M_k=m_0+\ldots+m_{k-1}$ and $N_k=n_0+\ldots+n_{k-1}$. 
We observe that 
$$
\psi_{\Omega,s,k}(\pi(x))=\psi_{D,s,k}(x)\text{\quad for all $x\in D$ and $k\in\N$}.
$$
In particular, we see
a $T_D$-invariant probability measure $\nu_D$ is an equilibrium state for $\Psi_{D,s}$
if and only if $\pi_*\nu_D$ is an equilibrium state for $\Psi_{\Omega,s}$. 
\color{black}
We make the following claims about $\Psi_{\Omega,s}$:
\begin{enumerate}
    \item $\Psi_{\Omega,s}$ is almost additive for each $s\in\R$. 
    \label{claim:almost additive}
    \item $\sum_{m,n=1}^\infty \sup_{\omega\in[(m,n)]}\psi_{\Omega,s,1}(\omega)<\infty$ for all
    $s\in\R$; \label{claim:summability}
    \item $P(\Psi_{\Omega,s})<\infty$ for all $s\in\R$;\label{claim:finite pressure}
    \item $s\mapsto P(\Psi_{\Omega,s})$ is convex;\label{claim:convex}
    \item $P(\Psi_{\Omega,s})>0$ if and only if $P(\Psi_s)>0$;\label{claim:T correspondence}
    \item $P(\Psi_{\Omega,0})=0$.\label{claim:zero pressure}
\end{enumerate}

Claim \eqref{claim:almost additive} follows from Lemma \ref{Almost additivity}.

For claim \eqref{claim:summability}, notice that $\psi_{\Omega,s,1}(\omega)=
e^{-(m+n)P}\|B_{m,n}\|^{s+t}$ for all $\omega\in[(m,n)]$. A simple calculation shows 
\begin{equation}
mn\leq \left\|\begin{bmatrix}1+mn & n \\ m & 1\end{bmatrix}\right\| \leq 4mn\label{eq:Bupperlower}
\end{equation}
for each $m,n \in \N$. 
Hence to establish \eqref{claim:summability}, it suffices to check that 
$$
\sum_{m,n=1}^\infty e^{-(m+n)P}(mn)^{s+t}<\infty.
$$
Since this quantity is the square of $\sum_{n=1}^\infty e^{-nP}n^{s+t}$ and $P>0$,
the claim holds. 

The deduction of \eqref{claim:finite pressure} from claims \eqref{claim:almost additive} and
\eqref{claim:summability} appears in \cite{IommiYayama2012}.
We give a self-contained proof. We have 
$$
\begin{aligned}
P(\Psi_{\Omega,s})
&=\lim_{k\to\infty}
\frac 1k\log\sum_{\mathbf m\in(\N^2)^k}e^{-(M_k+N_k)P}\|B_{m_{k-1},n_{k-1}}\cdots B_{m_0,n_0}\|^{s+t}\\
&\le\lim_{k\to\infty}
\frac 1k\log\sum_{\mathbf m\in(\N^2)^k}e^{-(M_k+N_k)P}\|B_{m_{k-1},n_{k-1}}\cdots B_{m_0,n_0}\|^{|s+t|}\\
&\le \lim_{k\to\infty}
\frac 1k\log\sum_{\mathbf m\in(\N^2)^k}\prod_{j=0}^{k-1}
e^{-(m_j+n_j)P}\|B_{m_j,n_j}\|^{|s+t|}\\
&=\lim_{k\to\infty}\frac 1k \log\left(\sum_{(m,n)\in \N^2}e^{-(m+n)P}\|B_{m,n}\|^{|s+t|}\right)^k\\
&=\log \sum_{(m,n)\in \N^2}e^{-(m+n)P}\|B_{m,n}\|^{|s+t|}\\
&\le\log\sum_{m,n=1}^\infty e^{-(m+n)P}(4mn)^{|s+t|}\\
&=2\log\sum_{n=1}^\infty e^{-nP}(2n)^{|s+t|}<\infty.
\end{aligned}
$$
In the sixth line, we used \eqref{eq:Bupperlower}.

For claim \eqref{claim:convex}, convexity of $s\mapsto P(\Psi_{\Omega,s})$
follows from a standard argument (e.g., see \cite[Section 3]{PU}) using 
H\"older's inequality and the
fact that $\psi_{\Omega,\alpha s+(1-\alpha)s',k}=\psi_{\Omega,s,k}^\alpha\psi_{\Omega,s',k}^{1-\alpha}$.

For claim \eqref{claim:T correspondence}, if $P(\Psi_s)>0$, by the variational principle, there is an ergodic
invariant measure $\mu$ on $\Sigma$ such that \[h_{\mu}(T)+\lim_{n\to\infty}
\frac 1n\int\log\psi_{s,n}\,d\mu>0.\] Using \eqref{AbramovKnill}, one can check
$$
h_{\mu_D}(T_{D})+\lim_{k\to\infty}\frac 1k\int\log\psi_{D,s,k}\,d\mu_D=
\frac{1}{\mu(D)}\left(h_{\mu}(T)+\lim_{n\to\infty}
\frac 1n\int\log\psi_{s,n}\,d\mu\right),
$$
where $\mu_D$ is the induced measure as usual, so that $h_{\mu_D}(T_{D})+\lim_{k\to\infty}
\frac 1k\int\log\psi_{D,s,k}\,d\mu_D>0$.
Pushing forward $\mu_D$ under the isomorphism to a measure on $\Omega$,
and using the variational principle again, we see $P(\Psi_{\Omega,s})>0$. 

Now suppose that $P(\Psi_s)\le 0$. By Theorem \ref{thm:IY}\eqref{it:VP}, in order to
show $P(\Psi_{\Omega,s})\le 0$,
it suffices to show that $h_{\nu}(T_{\Omega})+\lim\frac 1k\int \log\psi_{\Omega,s,k}\,d\nu\le 0$
for all invariant probability measures $\nu$ supported on a finite symbol full subshift 
of $\Omega$. Let $\nu$ be such a measure. 
Using the isomorphism $\pi$, the measure $\nu$ corresponds to a $T_D$-invariant measure $\nu_D$ on $D$. 
Since there are finitely many symbols and $r_D(x)=m+n$ if $x\in X_{m,n}$, we see that $r_D$ is bounded. 
Hence $\int r_D\,d\nu_D<\infty$. By Subsection \ref{subsec:induced}, $\nu_D$
lifts to an invariant probability measure $\mu$ on $\Sigma$. Since we assumed that 
$P(\Psi_s)\le 0$, it follows from the variational principle that 
$$
h_{\mu}(T)+\lim_{n\to\infty}\frac 1n\int\log\psi_{s,n}\,d\mu\le 0.
$$
Using \eqref{AbramovKnill} again, we see that 
$$
h_{\nu}(T_{\Omega})+\lim_{k\to\infty}\frac 1k\int\log\psi_{\Omega,s,k}\,d\nu\le 0.
$$
Hence $P(\Psi_{\Omega,s})\le 0$.

For claim \eqref{claim:zero pressure}, $s\mapsto P(\Psi_{\Omega,s})$ and $s\mapsto
P(\Psi_s)$ are convex (and hence continuous) functions defined for $s\in\R$. 
Applying claim \eqref{claim:T correspondence}, we see that $P(\Psi_{\Omega,s})>0$ for all
$s>0$ and $P(\Psi_{\Omega,s})\le 0$ for all $s\le 0$. It follows that $P(\Psi_{\Omega,0})=0$
as required.

We apply Theorem \ref{thm:IY}\eqref{it:Gibbs} to $\Psi_{\Omega,0}$. The hypotheses are verified by
claims \eqref{claim:almost additive} and \eqref{claim:summability}. Hence, there is a 
Gibbs equilibrium measure $\mu_{\Omega}$ for $\Psi_{\Omega,0}$. We further check that $\mu_{\Omega}$ 
is the unique equilibrium measure. It suffices to show that 
$h_{\mu_{\Omega}}(T_{\Omega})<\infty$. By the Gibbs property, 
$$
\mu_{\Omega}([(m ,n)]) \approx \|B_{m ,n}\|^te^{-(m+n)P}.
$$ 
In particular, $\mu_\Omega$ is fully supported on $\Omega$. 
By the above calculation, $\|B_{m ,n}\|^t \approx (mn)^t$. 
Since $P>0$, we see that $\mu_\Omega([(m,n)])$ decays exponentially.  This implies that 
the entropy of the generating 
partition $\{[(m,n)]\colon m,n\in\N^2\}$ is finite. 
Therefore, $h_{\mu_{\Omega}}(T_{\Omega})$ is finite. We have therefore 
established that there is a unique equilibrium measure on $(\Omega,T_{\Omega})$ for the potential $\Psi_{\Omega, 0}$.
It follows, using the isomorphism, that there is a unique equilibrium measure on $(D,T_D)$ for the potential $\Psi_{D,0}$.
We verify that this lifts to an invariant probability measure on $\Sigma$: we require
\[\int r_1\,d\mu_D<\infty.\] Using the isomorphism between $T_D$ and $T_{\Omega}\colon\Omega\to\Omega$,
this condition is equivalent to the condition
$\sum_{m,n}(m+n)\mu_\Omega[(m,n)]<\infty$. Since $P>0$, this is clearly satisfied as the terms 
decay exponentially. 

Since $\mu_D$ is an equilibrium measure for $\Psi_D$ and $P(\Psi_D)=0$, we have \[h_{\mu_D}(T_D)+
\lim_{k\to\infty}\frac 1k\int\psi_{D,0,k}\,d\mu_D=0.\]
By \eqref{AbramovKnill}, it follows that
$h_\mu(T)+\lim_{k\to\infty}\frac 1k\int\psi_{0,k}\,d\mu=0$. 
That is, $\mu$ is an equilibrium measure for $\Psi_0$. 
This argument shows that there is a bijection between ergodic $T_D$-invariant 
equilibrium measures on $D$ for $\Psi_{D,0}$ satisfying $\int r_1\,d\mu_D<\infty$
and ergodic $T$-invariant equilibrium measures $\mu$ on $X$ for $\Psi_0$ with the property that
$\mu(D)>0$.

Since $\mu_\Omega$ is the unique equilibrium measure for $\Psi_{\Omega,0}$, it follows that
$\mu$ is the unique ergodic equilibrium measure for $\Psi_0$ for which $\mu(D)>0$. 
However an ergodic measure for which $\mu(D)=0$ is either $\delta_{\bar 0}$ or 
$\delta_{\bar 1}$ and neither of these measures is an 
equilibrium measure for $\Psi_0$. Hence $\mu$ is the unique equilibrium measure for $\Psi_0$.
The fact that $\mu_\Omega$ is fully supported implies that $\mu$ is fully supported also. 

In the case $t=-2$, we have $P(-2\Phi_\A)=0$. The conditions for Theorem
\ref{thm:IY}\eqref{it:Gibbs} still hold, giving a Gibbs equilibrium measure $\mu_\Omega$ on
$\Omega$. Since $P=0$,
the cylinder sets have measure $\mu_\Omega([(m,n)])\approx(mn)^{-2}$. 
The finiteness of the entropy also holds so this is the unique
equilibrium measure $\mu_\Omega$ on $\Omega$ for the 
potential $\Psi_{\Omega,0}=-2\Phi_{\A_\Omega}$.
One can see, from the fact that $\sum_{m,n}(m+n)(mn)^{-2}=\infty$, 
that the expected return time is infinite.

Accordingly, as described in Subsection \ref{subsec:induced},
for $t=-2$, the equilibrium measure on $\Omega$ does not lift to an equilibrium measure
on $\Sigma$ and the only equilibrium measures for $-2\Phi_\A$ are $\delta_{\bar 0}$ and 
$\delta_{\bar 1}$.

\end{proof}

\bibliographystyle{abbrv}
\bibliography{Existence-of-Eq}
 
\end{document}